\documentclass[a4paper,10pt]{article}

\usepackage{amssymb}
\usepackage[centertags]{amsmath}
\usepackage{amsfonts}
\usepackage{amsthm}
\usepackage{MnSymbol} 
\newtheorem{theorem}{Theorem}[section]
\newtheorem{corollary}[theorem]{Corollary}

\newtheorem{conjecture}{Conjecture}

\usepackage{array,multirow,graphicx}
\usepackage{float}
\usepackage{tikz}
\usepackage{tikz-3dplot}
\usetikzlibrary{arrows.meta}
\usetikzlibrary{shapes,arrows}
\usepackage{wrapfig}
\tikzstyle{block}=[draw opacity=0.7,linewidth=1.4cm]{\usetikzlibrary{arrows,shapes}}
\usepackage{xcolor}
\usepackage{color, colortbl}
\usepackage{lscape}
\usepackage{multirow}
\definecolor{classicrose}{rgb}{0.98, 0.8, 0.91}
\definecolor{cottoncandy}{rgb}{1.0, 0.74, 0.85}
\definecolor{mediumchampagne}{rgb}{0.95, 0.9, 0.67}
\definecolor{maize}{rgb}{0.98, 0.93, 0.37}
\definecolor{celadon}{rgb}{0.67, 0.88, 0.69}
\definecolor{darkseagreen}{rgb}{0.56, 0.74, 0.56}
\definecolor{pastelyellow}{rgb}{0.99, 0.99, 0.59}
\definecolor{sandstorm}{rgb}{0.93, 0.84, 0.25}
\definecolor{skyblue}{rgb}{0.53, 0.81, 0.92}
\definecolor{royalblue}{rgb}{0.25, 0.41, 0.88}
\definecolor{celadon}{rgb}{0.67, 0.88, 0.69}
\definecolor{darkcyan}{rgb}{0.0, 0.55, 0.55}
\usepackage{multirow}
\usepackage{subcaption} 
\usepackage{chessboard} 

\usepackage{authblk} 
\usepackage{array}
\usepackage{array}
\newcolumntype{L}[1]{>{\raggedright\let\newline\\\arraybackslash\hspace{0pt}}m{#1}}
\newcolumntype{C}[1]{>{\centering\let\newline\\\arraybackslash\hspace{0pt}}m{#1}}
\newcolumntype{R}[1]{>{\raggedleft\let\newline\\\arraybackslash\hspace{0pt}}m{#1}}

\usepackage{geometry}
\title{Integer eigenvalues of the $n$-Queens graph}
\author[1,2]{Domingos M. Cardoso}
\author[1,2]{In\^es Ser\^odio Costa}
\author[1,2]{Rui Duarte}

\affil[1]{\small Centro de Investiga\c{c}\~{a}o e Desenvolvimento em Matem\'atica e Aplica\c{c}\~{o}es}
\affil[2]{\small Departamento de Matem\'atica, Universidade de  Aveiro, 3810-193, Aveiro, Portugal}

\begin{document}
\maketitle

\begin{abstract}
The $n$-Queens graph, $\mathcal{Q}(n)$, is the graph obtained from a $n\times n$ chessboard where each of its $n^2$ squares is a vertex and two vertices are adjacent if and only if they are in the same row, column or diagonal.
In a previous work the authors have shown that, for $n\ge4$, the least eigenvalue of $\mathcal{Q}(n)$ is $-4$ and its multiplicity is $(n-3)^2$. In this paper we prove that $n-4$ is also an eigenvalue of $\mathcal{Q}(n)$ and its multiplicity is at least $\frac{n+1}{2}$ or $\frac{n-2}{2}$ when $n$ is odd or even, respectively. Furthermore, when $n$ is odd, it is proved that $-3,-2\ldots,\frac{n-11}{2}$ and $\frac{n-5}{2},\ldots,n-5$ are additional integer eigenvalues of $\mathcal{Q}(n)$ and a family of eigenvectors associated with them is presented. Finally, conjectures about the multiplicity of the aforementioned eigenvalues and about the non-existence of any other integer eigenvalue are stated.
\end{abstract}

\medskip


\medskip


\medskip

\section{Introduction}\label{Sec-Introduction}
The $n$-Queens problem questions the possibility of placing $n$ non-attacking queens on a $n\times n$ chessboard. This problem has always solution for $n\ge4$ \cite{Pauls} and its solution is a maximum independent set of the $n$-Queens graph. Some historical notes about this problem are available in \cite{2009BellStevens} and \cite{1977Campbell}.

The $n$-Queens graph, $\mathcal{Q}(n)$, is a graph associated to the $n\times n$ chessboard, $\mathcal{T}_n$, it has $n^2$ vertices, each one corresponding to a square of $\mathcal{T}_n$, and two distinct vertices are adjacent if and only if the corresponding squares in $\mathcal{T}_n$ belong to the same row, column, or diagonal. This graph represents all legal moves of the queen on a chessboard. In Figure \ref{fig_Q4}, $\mathcal{Q}(4)$ is presented.

\begin{figure}[h!]
	\centering
	\begin{tikzpicture}
		\begin{scope}[every node/.style={circle,thick,draw,minimum size=0.8cm}]
			\node (1) at (0,0) {\tiny(1,1)};
			\node (2) at (1.5,0) {\tiny(1,2)};
			\node (3) at (3,0) {\tiny(1,3)};
			\node (4) at (4.5,0) {\tiny(1,4)};
			\node (5) at (0,-1.5) {\tiny(2,1)};
			\node (6) at (1.5,-1.5) {\tiny(2,2)};
			\node (7) at (3,-1.5) {\tiny(2,3)};
			\node (8) at (4.5,-1.5) {\tiny(2,4)};
			\node (9) at (0,-3) {\tiny(3,1)};
			\node (10) at (1.5,-3) {\tiny(3,2)};
			\node (11) at (3,-3) {\tiny(3,3)};
			\node (12) at (4.5,-3) {\tiny(3,4)};
			\node (13) at (0,-4.5) {\tiny(4,1)};
			\node (14) at (1.5,-4.5) {\tiny(4,2)};
			\node (15) at (3,-4.5) {\tiny(4,3)};
			\node (16) at (4.5,-4.5) {\tiny(4,4)};
		\end{scope}
		
		\begin{scope}[>={Stealth[black]},
			every edge/.style={draw=black,line width=0.5pt}]
			\path (1) edge node {} (2);
			\path (1) edge node {} (5);
			\path (1) edge node {} (6);
			\path (1) edge[bend left=25] node {} (3);
			\path (1) edge[bend left=25] node {} (4);
			\path (1) edge[bend right=25] node {} (9);
			\path (1) edge[bend right=25] node {} (13);
			\path (1) edge[bend left=18] node {} (11);
			\path (1) edge[bend left=18] node {} (16);
			
			\path (2) edge node {} (3);
			\path (2) edge node {} (5);
			\path (2) edge node {} (6);
			\path (2) edge node {} (7);
			\path (2) edge[bend left=25] node {} (4);
			\path (2) edge[bend left=25] node {} (12);
			\path (2) edge[bend right=25] node {} (10);
			\path (2) edge[bend right=25] node {} (14);
			
			\path (3) edge node {} (6);
			\path (3) edge node {} (7);
			\path (3) edge node {} (8);
			\path (3) edge node {} (4);
			\path (3) edge[bend right=25] node {} (9);
			\path (3) edge[bend left=25] node {} (11);
			\path (3) edge[bend left=25] node {} (15);
			
			\path (4) edge node {} (7);
			\path (4) edge node {} (8);
			\path (4) edge[bend left=18] node {} (10);
			\path (4) edge[bend left=18] node {} (13);
			\path (4) edge[bend left=25] node {} (12);
			\path (4) edge[bend left=25] node {} (16);
			
			\path (5) edge node {} (6);
			\path (5) edge node {} (9);
			\path (5) edge node {} (10);
			\path (5) edge[bend right=25] node {} (13);
			\path (5) edge[bend right=25] node {} (15);
			\path (5) edge[bend left=25] node {} (7);
			\path (5) edge[bend left=25] node {} (8);
			
			\path (6) edge node {} (7);
			\path (6) edge node {} (9);
			\path (6) edge node {} (10);
			\path (6) edge node {} (11);
			\path (6) edge[bend right=25] node {} (14);
			\path (6) edge[bend left=25] node {} (8);
			\path (6) edge[bend left=18] node {} (16);
			
			\path (7) edge node {} (11);
			\path (7) edge node {} (8);
			\path (7) edge node {} (10);
			\path (7) edge node {} (12);
			\path (7) edge[bend left=18] node {} (13);
			\path (7) edge[bend left=25] node {} (15);
			
			\path (8) edge node {} (11);
			\path (8) edge node {} (12);
			\path (8) edge[bend left=25] node {} (14);
			\path (8) edge[bend left=25] node {} (16);
			
			\path (9) edge node {} (10);
			\path (9) edge node {} (13);
			\path (9) edge node {} (14);
			\path (9) edge[bend right=25] node {} (11);
			\path (9) edge[bend right=25] node {} (12);
			
			\path (10) edge node {} (11);
			\path (10) edge node {} (13);
			\path (10) edge node {} (14);
			\path (10) edge node {} (15);
			\path (10) edge[bend right=25] node {} (12);
			
			\path (11) edge node {} (12);
			\path (11) edge node {} (14);
			\path (11) edge node {} (15);
			\path (11) edge node {} (16);
			
			\path (12) edge node {} (15);
			\path (12) edge node {} (16);
			
			\path (13) edge node {} (14);
			\path (13) edge[bend right=25] node {} (15);
			\path (13) edge[bend right=25] node {} (16);
			\path (14) edge node {} (15);
			\path (14) edge[bend right=25] node {} (16);
			\path (15) edge node {} (16);			
		\end{scope}
	\end{tikzpicture}
	\caption{$\mathcal{Q}(4)$.}\label{fig_Q4}
\end{figure}

Throughout the text $A$ denotes the adjacency matrix of the graph $\mathcal{Q}(n)$ and its eigenvalues are also called the eigenvalues of the graph. The eigenspace associated with an eigenvalue $\mu$ of $\mathcal{Q}(n)$ is denoted by $\mathcal{E}_{\mathcal{Q}(n)} (\mu)$.

The least eigenvalue of the $n$-Queens graph and its multiplicity are already known \cite{2022CardosoCostaDuarte} and such results are recalled in Section~\ref{Sec-LeastEigenvalue}. In Section~\ref{Sec-SequenceEigenvalues} it is proved that $-3, -2, \ldots, \frac{n-11}{2}, \frac{n-5}{2},$ $ \ldots, n-5, n-4$ are eigenvalues of $\mathcal{Q}(n)$, when $n \ge 4$
is odd and a family of eigenvectors associated with them is introduced. In Section~\ref{Sec-n-4}, it is proved that $n-4$ is also an eigenvalue of $\mathcal{Q}(n)$ when $n \ge 3$. Additionally, a family of $\frac{n+1}{2}$ (resp. $\frac{n-2}{2}$) linearly independent eigenvectors associated to $n-4$ is presented when $n$ is odd (resp. even). We finish the paper with some conjectures regarding the multiplicities of the eigenvalues, $-3,-2,\ldots,\frac{n-11}{2},\frac{n-5}{2},\ldots,n-5,n-4$, and the non-existence of any other integer eigenvalues.\\

The rows and columns of $\mathcal{T}_n$ are labeled from top to the bottom and from left to right, respectively, and we use the ordered pair $(i,j)$ to refer to the vertex of $\mathcal{Q}(n)$, which is in the $i$-th row and $j$-th column.
For an easier representation of the vectors, in this paper, they will be written over the chessboard. Thus, for a vector $v$, the coordinate corresponding to the vertex $(i,j)$ will be displayed at the square $(i,j)$ of the chessboard and it will be denoted by $v_{(i,j)}$.

\section{Least eigenvalue of $\mathcal{Q}(n)$} \label{Sec-LeastEigenvalue}

In \cite{2022CardosoCostaDuarte}, using the combinatorial definition of \emph{Edge Clique Partition} (ECP) of a graph $G$, it was proved that the symmetric of the maximum clique degree of an ECP relative to $G$ is a lower bound for the eigenvalues of a graph.
In particular, for the $n$-Queens graph, with $n \ge 4$, the ECP with maximal cliques in all the four ``edge directions'' was considered and the results of this section were obtained.

\medskip

Consider the family of vectors
$$
\mathbf{V}_n = \{X_n^{(a,b)} \in \mathbb{R}^{n^2} \mid (a,b) \in [n-3]^2 \},
$$
with $n \ge 4$, where $X_n^{(a,b)}$ is the vector defined by

\begin{equation}\label{ev_components}
	\big[X_n^{(a,b)}\big]_{(i,j)} = \begin{cases}
            \big[X_4\big]_{(i-a+1,j-b+1)}, & \text{ if } a \leq i \leq a+3 \text{ and } b \leq j \leq b+3;\\
		0,                             & \text{otherwise,}
	\end{cases}
\end{equation}
with $X_4$ is the vector presented in Figure~\ref{vector_X4}.

\begin{figure}[h!]
	\[
	\begin{array}{|c|c|c|c|} \hline
		0          & \textbf{1} & \textbf{-1}& 0           \\ \hline
		\textbf{-1}& 0          & 0          & \textbf{1}  \\ \hline
		\textbf{1} & 0          & 0          & \textbf{-1} \\ \hline
		0          &\textbf{-1} & \textbf{1} & 0           \\ \hline
	\end{array}
	\]
	\caption{The vector $X_4$.}\label{vector_X4}
\end{figure}

This family $\mathbf{V}_n$ is formed by $(n-3)^2$ linearly independent vectors. The four vectors of $\mathbf{V}_5$ are depicted in Figure~\ref{vetors_11-12-21-22}.

\begin{figure}[h!]\centering
	$\begin{array}{cc}
 \overbrace{
		\begin{array}{|c|c|c|c|c|} \hline
			0          & \textbf{1} & \textbf{-1}& 0          & 0 \\ \hline
			\textbf{-1}& 0          & 0          & \textbf{1} & 0 \\ \hline
			\textbf{1} & 0          & 0          & \textbf{-1}& 0 \\ \hline
			0          &\textbf{-1} & \textbf{1} & 0          & 0 \\ \hline
			0          & 0          & 0          & 0          & 0 \\ \hline
		\end{array}
  }^{X_5^{(1,1)}}
		&
  \overbrace{
		\begin{array}{|c|c|c|c|c|} \hline
			0 & 0          & \textbf{1} & \textbf{-1}& 0 \\ \hline
			0 & \textbf{-1}& 0          & 0          & \textbf{1} \\ \hline
			0 & \textbf{1} & 0          & 0          & \textbf{-1} \\ \hline
			0 & 0          & \textbf{-1}& \textbf{1} & 0 \\ \hline
			0 & 0          & 0          & 0          & 0 \\ \hline
		\end{array}}^{X_5^{(1,2)}}
		\\ \\
 \underbrace{
		\begin{array}{|c|c|c|c|c|}\hline
			0          & 0          & 0          & 0          & 0 \\ \hline
			0          & \textbf{1} & \textbf{-1}& 0          & 0\\ \hline
			\textbf{-1}& 0          & 0          & \textbf{1} & 0 \\ \hline
			\textbf{1} & 0          & 0          & \textbf{-1}& 0 \\ \hline
			0          & \textbf{-1}& \textbf{1} & 0          & 0 \\ \hline
		\end{array}
  }_{X_5^{(2,1)}}
  &
  \underbrace{
		\begin{array}{|c|c|c|c|c|}\hline
			0 & 0          & 0          & 0          & 0  \\ \hline
			0 & 0          & \textbf{1} & \textbf{-1}& 0 \\ \hline
			0 & \textbf{-1}& 0          & 0          & \textbf{1} \\ \hline
			0 & \textbf{1} & 0          & 0          & \textbf{-1} \\ \hline
			0 & 0          & \textbf{-1}& \textbf{1} & 0 \\ \hline
		\end{array}
  }_{X_5^{(2,2)}}
	\end{array}
 $
	\caption{The vectors defined in \eqref{ev_components} for $n=5$.} \label{vetors_11-12-21-22}
\end{figure}

\medskip

Once these vectors are known, the following result appears.

\begin{theorem}\cite[Theorem 4.3.]{2022CardosoCostaDuarte} \label{th-4eigenvalue}
	$-4$ is an eigenvalue of $\mathcal{Q}(n)$ with multiplicity $(n-3)^2$ and $\mathbf{V}_n$ is a basis for
	$\mathcal{E}_{\mathcal{Q}(n)} (-4)$.
\end{theorem}

From now on, the zero entries of eigenvectors are omitted.

\section{A sequence of integer eigenvalues} \label{Sec-SequenceEigenvalues}
In this section, we assume that $n\ge4$ is odd and $\lambda \in \mathbb{Z}$ is such that $1 \leq \lambda+4 \leq n$. Given $i,j \in [n]$, let $i \ominus j = |i-j|$ and $i \oplus j = |i+j - (n+1)|$. Let $k = (\lambda+4) - \frac{n-1}{2}$ and $P_{n,\lambda}, Q_{n,\lambda} \in \mathbb{R}^{n \times n}$ the vectors defined as follows.

\begin{align*}
	(P_{n,\lambda})_{(i,j)} & =
	\begin{cases}
		k, & \text{if $i \ominus j = n - (\lambda+4)$;} \\
		1, & \text{if $i \ominus j < n - (\lambda+4)$,} \\
		& i \oplus j < \lambda + 4, \\
		& \text{and $i \ominus j \equiv n - (\lambda+4) \mod{2}$}; \\
		0, & \text{otherwise};
	\end{cases}
	\intertext{and}
	(Q_{n,\lambda})_{(i,j)} & =
	\begin{cases}
		-k, & \text{if $i \oplus j = n - (\lambda+4)$;} \\
		-1, & \text{if $i \oplus j < n - (\lambda+4)$,} \\
		& i \ominus j < \lambda + 4, \\
		& \text{and $i \oplus j \equiv n - (\lambda+4) \mod{2}$}; \\
		0, & \text{otherwise};
	\end{cases} \\
	& = - (P_{n,\lambda})_{(n+1-i,j)} = - (P_{n,\lambda})_{(i,n+1-j)} .
\end{align*}

In Figure \ref{fig:PeQ} we display $P_{11,-1}$ and $Q_{11,-1}$. Note that if $n=11$ and $\lambda=-1$, $k=-2$.

\begin{figure}[h!]
	\centering
	
	{\scriptsize
			\begin{tabular}{|c|c|c|c|c|c|c|c|c|c|c|}
				\hline & & & & & \phantom{-1} & & & \cellcolor{royalblue}-2 & & \\
				\hline & & & & & & & \cellcolor{skyblue}1 & & \cellcolor{royalblue}-2 & \\
				\hline & & & & & & \cellcolor{skyblue}1 & & \cellcolor{skyblue}1 & & \cellcolor{royalblue}-2 \\
				\hline & & & & & \cellcolor{skyblue}1 & & \cellcolor{skyblue}1 & & \cellcolor{skyblue}1 & \\
				\hline & & & & \cellcolor{skyblue}1 & & \cellcolor{skyblue}1 & & \cellcolor{skyblue}1 & & \\
				\hline & & & \cellcolor{skyblue}1 & & \cellcolor{skyblue}1 & & \cellcolor{skyblue}1 & & & \\
				\hline & & \cellcolor{skyblue}1 & & \cellcolor{skyblue}1 & & \cellcolor{skyblue}1 & & & & \\
				\hline & \cellcolor{skyblue}1 & & \cellcolor{skyblue}1 & & \cellcolor{skyblue}1 & & & & & \\
				\hline \cellcolor{royalblue}-2 & & \cellcolor{skyblue}1 & & \cellcolor{skyblue}1 & & & & & & \\
				\hline & \cellcolor{royalblue}-2 & & \cellcolor{skyblue}1 & & & & & & & \\
				\hline & & \cellcolor{royalblue}-2 & & & & & & & & \\
				\hline
			\end{tabular}\\\vspace{0.5cm}
			\begin{tabular}{|c|c|c|c|c|c|c|c|c|c|c|}
				\hline & & \cellcolor{sandstorm}2 & & & \phantom{-1} & & & & & \\
				\hline & \cellcolor{sandstorm}2 & & \cellcolor{pastelyellow}-1 & & & & & & & \\
				\hline \cellcolor{sandstorm}2 & & \cellcolor{pastelyellow}-1 & & \cellcolor{pastelyellow}-1 & & & & & & \\
				\hline & \cellcolor{pastelyellow}-1 & & \cellcolor{pastelyellow}-1 & & \cellcolor{pastelyellow}-1 & & & & & \\
				\hline & & \cellcolor{pastelyellow}-1 & & \cellcolor{pastelyellow}-1 & & \cellcolor{pastelyellow}-1 & & & & \\
				\hline & & & \cellcolor{pastelyellow}-1 & & \cellcolor{pastelyellow}-1 & & \cellcolor{pastelyellow}-1 & & & \\
				\hline & & & & \cellcolor{pastelyellow}-1 & & \cellcolor{pastelyellow}-1 & & \cellcolor{pastelyellow}-1 & & \\
				\hline & & & & & \cellcolor{pastelyellow}-1 & & \cellcolor{pastelyellow}-1 & & \cellcolor{pastelyellow}-1 & \\
				\hline & & & & & & \cellcolor{pastelyellow}-1 & & \cellcolor{pastelyellow}-1 & & \cellcolor{sandstorm}2 \\
				\hline & & & & & & & \cellcolor{pastelyellow}-1 & & \cellcolor{sandstorm}2 & \\
				\hline & & & & & & & & \cellcolor{sandstorm}2 & & \\
				\hline
			\end{tabular}
	}
	\caption{The vectors $P_{11,-1}$ (table on the top) and $Q_{11,-1}$ (table on the bottom).}
	\label{fig:PeQ}
\end{figure}

\medskip

Now, let us consider the family of vectors
\begin{equation}\label{family_vectors}
E_{n,\lambda} = P_{n,\lambda} + Q_{n,\lambda}.
\end{equation}

\begin{theorem}
	$A E_{n,\lambda} = \lambda E_{n,\lambda}$.
\end{theorem}

\begin{proof}
	We prove that $(A + 4 I_n) E_{n,\lambda} = (\lambda + 4) E_{n,\lambda}$.
	
	First note that for every $\ell \in [n]$,
	\[
	\sum_{j=1}^n (P_{n,\lambda})_{(\ell,j)} = - \sum_{j=1}^n (Q_{n,\lambda})_{(\ell,j)} \quad \text{and} \quad \sum_{i=1}^n (P_{n,\lambda})_{(i,\ell)} = - \sum_{i=1}^n (Q_{n,\lambda})_{(i,\ell)},
	\]
	so
	\[
	\sum_{j=1}^n (E_{n,\lambda})_{(\ell,j)} = 0 = \sum_{i=1}^n (E_{n,\lambda})_{(i,\ell)}.
	\]
	
	Second note that
		\begin{align*}
            ((A + 4 I_n) E_{n,\lambda})_{(p,q)} = & \sum \limits_{(i,j) \sim (p,q)} (E_{n,\lambda})_{(i,j)} + 4 (E_{n,\lambda})_{(p,q)}\\
			= & \sum_{j=1}^n (E_{n,\lambda})_{(p,j)} + \sum_{i=1}^n (E_{n,\lambda})_{(i,q)} + \\
			& + \sum_{i+j=p+q}(E_{n,\lambda})_{(i,j)} + \sum_{i-j=p-q} (E_{n,\lambda})_{(i,j)}\\
			= & \sum_{i+j=p+q} (P_{n,\lambda})_{(i,j)} + \sum_{i-j=p-q} (P_{n,\lambda})_{(i,j)} + \\
			& + \sum_{i+j=p+q} (Q_{n,\lambda})_{(i,j)} + \sum_{i-j=p-q} (Q_{n,\lambda})_{(i,j)}.
		\end{align*}

For convenience, the parameters
	 	\begin{eqnarray*}
	 		\alpha&=(\lambda + 4) (E_{n,\lambda})_{(p,q)}, \quad \beta&= \sum \limits_{i+j=p+q} (P_{n,\lambda})_{(i,j)},\\
	 		\gamma&=\sum \limits_{i-j=p-q} (P_{n,\lambda})_{(i,j)}, \quad \delta&=\sum \limits_{i+j=p+q} (Q_{n,\lambda})_{(i,j)} \text{ and}\\
	 		\epsilon&= \sum \limits_{i-j=p-q} (Q_{n,\lambda})_{(i,j)}, \quad &
	 	\end{eqnarray*}
will be used in Table~\ref{parameters_table} and Table~\ref{tbl4}.\\

	\noindent \textbf{Case 1:} $\lambda+4 = n$. \\
	For every $\ell \in \{ 0 \} \cup [n]$,
	\begin{align*}
		\sum_{i \oplus j = \ell} (P_{n,\lambda})_{(i,j)} & =
		\begin{cases}
			k & \text{if $\ell$ is even}, \\
			0 & \text{otherwise,}
		\end{cases} = - \sum_{i \ominus j = \ell} (Q_{n,\lambda})_{(i,j)}
		\intertext{and}
		\sum_{i \ominus j = \ell} (P_{n,\lambda})_{(i,j)} & =
		\begin{cases}
			kn & \text{if $\ell = 0$,} \\
			0 & \text{otherwise}
		\end{cases} = - \sum_{i \oplus j = \ell} (Q_{n,\lambda})_{(i,j)} .
	\end{align*}

\begin{table}[h!]
    \centering
		\[
		{
			\begin{array}{|c|c|c|c|c|c|c|c|}
			
				\hline & p \ominus q & p \oplus q & \alpha & \beta & \gamma & \delta & \epsilon \\
				\hline
				\hline \multirow{4}{*}{$p+q$ even} & = 0 & = 0 & 0 & k & nk & -nk & -k \\
				& = 0 & > 0 & nk & k & nk & 0 & -k \\
				& > 0 & = 0 & n(-k) & k & 0 & -nk & -k \\
				& > 0 & > 0 & 0 & k & 0 & 0 & -k \\
				\hline \text{$p+q$ odd} & \neq 0 & \neq 0 & 0 & 0 & 0 & 0 & 0 \\
				\hline
			\end{array}
		}
		\]
    \caption{All possible values of $\alpha,\beta, \gamma, \delta$ and $\epsilon$, when $\lambda+4 = n$. }\label{parameters_table}
\end{table}

		\medskip
	
	\noindent \textbf{Case 2:} $\lambda+4 < n$. \\
	For every $\ell \in \{ 0 \} \cup [n]$,

	\begin{align*}
		\sum_{i \oplus j = \ell} (P_{n,\lambda})_{(i,j)} & =
		(\lambda+4) \cdot
		\begin{cases}
			1, & \text{if $\ell < \lambda + 4$ and} \\
			& \ell \not\equiv \lambda + 4 \bmod{2}; \\
			0 & \text{otherwise};
		\end{cases}
		\ = - \sum_{i \ominus j = \ell} (Q_{n,\lambda})_{(i,j)}
		\intertext{and}
		\sum_{i \ominus j = \ell} (P_{n,\lambda})_{(i,j)} & =
		(\lambda+4) \cdot
		\begin{cases}
			k, & \text{if $\ell = n - (\lambda+4)$}; \\
			1, & \text{if $\ell < n - (\lambda+4)$ and} \\
			& \ell \not\equiv \lambda + 4 \bmod{2}; \\
			0, & \text{otherwise};
		\end{cases}
		\ = - \sum_{i \oplus j = \ell} (Q_{n,\lambda})_{(i,j)}.
	\end{align*}
	
	Taking into account Table~\ref{tbl4}, where it is considered the subcases $n>2(\lambda+4)$ and $n<2(\lambda+4$), which correspond to $\lambda$ being $-3,\ldots,\frac{n-9}{2}$ and $\frac{n-7}{2},\ldots,-5$, respectively, it follows that $E_{n,\lambda}$ is an eigenvector of $\mathcal{Q}(n)$ associated with $\lambda$.
\end{proof}
	
	\begin{table}[h!]
		\[    \!\!\!\!\!\!\!\!\!\!\!\!\!\!\!\!\!\!\!\!\!\!\!\!\!\!\!\!\!\!\!\!\!\!\!\!
		{\footnotesize
			{\renewcommand{\arraystretch}{1.22}
				\begin{array}{|c|c|c|c|c|c|c|c|}
					\hline\multicolumn{8}{|c|}{\textbf{Subcase 2.1: } n > 2 (\lambda + 4)}\\
					\hline & \underbrace{p \ominus q}_{(a)} & \underbrace{p \oplus q}_{(b)} & \alpha & \beta & \gamma & \delta & \epsilon \\
					\hline
					\hline \parbox[t]{2mm}{\multirow{16}{*}{\rotatebox[origin=c]{90}{$p+q \not\equiv \lambda+4 \bmod{2}$}}} & \multirow{4}{*}{$<\lambda+4$} & <\lambda+4 & (\lambda + 4)(1-1) & \lambda+4 & \lambda+4 & (\lambda+4)(-1) & (\lambda+4)(-1) \\
					&  & \lambda+4 \le (b) < n-(\lambda+4) & (\lambda + 4)(0-1) & 0 & \lambda+4 & (\lambda+4)(-1) & (\lambda+4)(-1) \\
					&  & = n-(\lambda+4) & (\lambda + 4)(0-k) & 0 & \lambda+4 & (\lambda+4)(-k) & (\lambda+4)(-1) \\
					&  & > n-(\lambda+4) & (\lambda + 4)(0-0) & 0 & \lambda+4 & 0 & (\lambda+4)(-1) \\
					\cline{2-8}
					& \multirow{4}{*}{$\lambda+4 \le (a) < n-(\lambda+4)$} & <\lambda+4 & (\lambda + 4)(1-0) & \lambda+4 & \lambda+4 & (\lambda+4)(-1) & 0 \\
					&  & \lambda+4 \le (b) < n-(\lambda+4) & (\lambda + 4)(0-0) & 0 & \lambda+4 & (\lambda+4)(-1) & 0 \\
					&  & = n-(\lambda+4) & --- & --- & --- & --- & --- \\
					&  & > n-(\lambda+4) & --- & --- & --- & --- & --- \\
					\cline{2-8}
					& \multirow{4}{*}{$= n-(\lambda+4)$} & <\lambda+4 & (\lambda + 4)(k-0) & \lambda+4 & (\lambda+4)k & (\lambda+4)(-1) & 0 \\
					&  & = n-(\lambda+4) & --- & --- & --- & --- & --- \\
					&  & \lambda+4 \le (b) < n-(\lambda+4) & --- & --- & --- & --- & --- \\
					&  & > n-(\lambda+4) & --- & --- & --- & --- & --- \\
					\cline{2-8}
					& \multirow{4}{*}{$> n-(\lambda+4)$} & < \lambda+4 & (\lambda + 4)(0-0) & \lambda+4 & 0 & (\lambda+4)(-1) & 0 \\
					&  & \lambda+4 \le (b) < n-(\lambda+4) & --- & --- & --- & --- & --- \\
					&  & = n-(\lambda+4) & --- & --- & --- & --- & --- \\
					&  & > n-(\lambda+4) & --- & --- & --- & --- & --- \\\hline
					\parbox[t]{2mm}{\rotatebox[origin=c]{90}{\text{\scriptsize { } otherwise { }}}} & \multicolumn{2}{c|}{\text{ any case }} & 0 & 0 & 0 & 0 & 0 \\
					\hline
				\end{array}
			}
		}
		\]
		\[    \!\!\!\!\!\!\!\!\!\!\!\!\!\!\!\!\!\!\!\!\!\!\!\!\!\!\!\!\!\!\!\!\!\!\!\!
		{\footnotesize
			{\renewcommand{\arraystretch}{1.22}
				\begin{array}{|c|c|c|c|c|c|c|c|}
					\hline\multicolumn{8}{|c|}{\textbf{Subcase 2.2: } n < 2 (\lambda + 4)}\\
					\hline & \underbrace{p \ominus q}_{(a)} & \underbrace{p \oplus q}_{(b)} & \alpha & \beta & \gamma & \delta & \epsilon \\
					\hline
					\hline \parbox[t]{2mm}{\multirow{16}{*}{\rotatebox[origin=c]{90}{$p+q \not\equiv \lambda+4 \bmod{2}$}}} & \multirow{4}{*}{$< n-(\lambda+4)$} & < n-(\lambda+4) & (\lambda+4)(1-1) & \lambda+4 & \lambda+4 & (\lambda+4)(-1) & (\lambda+4)(-1) \\
					&  & = n-(\lambda+4) & (\lambda+4)(1-k) & \lambda+4 & \lambda+4 & (\lambda+4)(-k) & (\lambda+4)(-1) \\
					&  & n-(\lambda+4) < (b) < \lambda+4 & (\lambda+4)(1-0) & \lambda+4 & \lambda+4 & 0 & (\lambda+4)(-1) \\
					&  & \geq \lambda+4 & (\lambda+4)(0-0) & 0 & \lambda+4 & 0 & (\lambda+4)(-1) \\
					\cline{2-8}
					& \multirow{4}{*}{$= n-(\lambda+4)$} & < n-(\lambda+4) & (\lambda+4)(k-1) & \lambda+4 & (\lambda+4)k & (\lambda+4)(-1) & (\lambda+4)(-1) \\
					&  & = n-(\lambda+4) & (\lambda+4)(k-k) & \lambda+4 & (\lambda+4)k & (\lambda+4)(-k) & (\lambda+4)(-1) \\
					&  & n-(\lambda+4) < (b) < \lambda+4 & (\lambda+4)(k-0) & \lambda+4 & (\lambda+4)k & 0 & (\lambda+4)(-1) \\
					&  & \geq \lambda+4 & --- & --- & --- & --- & --- \\
					\cline{2-8}
					& \multirow{4}{*}{$n-(\lambda+4) < (a) < \lambda+4$} & < n-(\lambda+4) & (\lambda+4)(0-1) & \lambda+4 & 0 & (\lambda+4)(-1) & (\lambda+4)(-1) \\
					&  & = n-(\lambda+4) & (\lambda+4)(0-k) & \lambda+4 & 0 & (\lambda+4)(-k) & (\lambda+4)(-1) \\
					&  & n-(\lambda+4) < (b) < \lambda+4 & (\lambda+4)(0-0) & \lambda+4 & 0 & 0 & (\lambda+4)(-1) \\
					&  & \geq \lambda+4 & --- & --- & --- & --- & --- \\
					\cline{2-8}
					& \multirow{4}{*}{$\geq \lambda+4$} & < n-(\lambda+4) & (\lambda+4)(0-0) & \lambda+4 & 0 & (\lambda+4)(-1) & 0 \\
					&  & = n-(\lambda+4) & --- & --- & --- & --- & --- \\
					&  & n-(\lambda+4) < (b) < \lambda+4 & --- & --- & --- & --- & --- \\
					&  & \geq \lambda+4 & --- & --- & --- & --- & --- \\\hline
					\parbox[t]{2mm}{\rotatebox[origin=c]{90}{\text{\scriptsize { } otherwise { }}}} & \multicolumn{2}{c|}{\text{ any case }} & 0 & 0 & 0 & 0 & 0 \\
					\hline
				\end{array}
			}
		}
		\]
		\caption{Subcases of Case 2.}\label{tbl4}
	\end{table}

Note that $P_{n, \frac{n-9}{2}} = -Q_{n, \frac{n-9}{2}}$ and $P_{n, \frac{n-7}{2}} = -Q_{n, \frac{n-7}{2}}$. In Figure~\ref{fig:PeQ2} we present $P_{11, 1}$ and $P_{11, 2}$.

\begin{figure}[h!]
	\centering
	{\scriptsize
			\begin{tabular}{|c|c|c|c|c|c|c|c|c|c|c|}
				\hline & & & & & & \cellcolor{royalblue}0 & & & & \\
				\hline & & & & & \cellcolor{skyblue}1 & & \cellcolor{royalblue}0 & & & \\
				\hline & & & & \cellcolor{skyblue}1 & & \cellcolor{skyblue}1 & & \cellcolor{royalblue}0 & & \\
				\hline & & & \cellcolor{skyblue}1 & & \cellcolor{skyblue}1 & & \cellcolor{skyblue}1 & & \cellcolor{royalblue}0 & \\
				\hline & & \cellcolor{skyblue}1 & & \cellcolor{skyblue}1 & & \cellcolor{skyblue}1 & & \cellcolor{skyblue}1 & & \cellcolor{royalblue}0 \\
				\hline & \cellcolor{skyblue}1 & & \cellcolor{skyblue}1 & & \cellcolor{skyblue}1 & & \cellcolor{skyblue}1 & & \cellcolor{skyblue}1 & \\
				\hline \cellcolor{royalblue}0 & & \cellcolor{skyblue}1 & & \cellcolor{skyblue}1 & & \cellcolor{skyblue}1 & & \cellcolor{skyblue}1 & & \\
				\hline & \cellcolor{royalblue}0 & & \cellcolor{skyblue}1 & & \cellcolor{skyblue}1 & & \cellcolor{skyblue}1 & & & \\
				\hline & & \cellcolor{royalblue}0 & & \cellcolor{skyblue}1 & & \cellcolor{skyblue}1 & & & & \\
				\hline & & & \cellcolor{royalblue}0 & & \cellcolor{skyblue}1 & & & & & \\
				\hline & & & & \cellcolor{royalblue}0 & & & & & & \\
				\hline
			\end{tabular}\\\vspace{0.5cm}
			\begin{tabular}{|c|c|c|c|c|c|c|c|c|c|c|}
				\hline & & & & & \cellcolor{royalblue}1 & & & & & \\
				\hline & & & & \cellcolor{skyblue}1 & & \cellcolor{royalblue}1 & & & & \\
				\hline & & & \cellcolor{skyblue}1 & & \cellcolor{skyblue}1 & & \cellcolor{royalblue}1 & & & \\
				\hline & & \cellcolor{skyblue}1 & & \cellcolor{skyblue}1 & & \cellcolor{skyblue}1 & & \cellcolor{royalblue}1 & & \\
				\hline & \cellcolor{skyblue}1 & & \cellcolor{skyblue}1 & & \cellcolor{skyblue}1 & & \cellcolor{skyblue}1 & & \cellcolor{royalblue}1 & \\
				\hline \cellcolor{royalblue}1 & & \cellcolor{skyblue}1 & & \cellcolor{skyblue}1 & & \cellcolor{skyblue}1 & & \cellcolor{skyblue}1 & & \cellcolor{royalblue}1 \\
				\hline & \cellcolor{royalblue}1 & & \cellcolor{skyblue}1 & & \cellcolor{skyblue}1 & & \cellcolor{skyblue}1 & & \cellcolor{skyblue}1 & \\
				\hline & & \cellcolor{royalblue}1 & & \cellcolor{skyblue}1 & & \cellcolor{skyblue}1 & & \cellcolor{skyblue}1 & & \\
				\hline & & & \cellcolor{royalblue}1 & & \cellcolor{skyblue}1 & & \cellcolor{skyblue}1 & & & \\
				\hline & & & & \cellcolor{royalblue}1 & & \cellcolor{skyblue}1 & & & & \\
				\hline & & & & & \cellcolor{royalblue}1 & & & & & \\
				\hline
			\end{tabular}
	}
	\caption{$P_{11, 1}$ and $P_{11, 2}$}
	\label{fig:PeQ2}
\end{figure}

Since $E_{n,\lambda}$ is the zero vector if $\lambda \in \{ \frac{n-9}{2}, \frac{n-7}{2} \}$ and a non-zero vector if $\lambda \in \{ -3, \ldots, \frac{n-11}{2} \} \cup \{ \frac{n-5}{2}, \ldots, n-4 \}$, we have the following result.

\begin{corollary}
	For every odd $n \ge 4$ and $\lambda \in \{ -3, \ldots, \frac{n-11}{2} \} \cup \{ \frac{n-5}{2}, \ldots, n-4 \}$, $E_{n,\lambda}$ is an eigenvector of $\mathcal{Q}(n)$ associated with the eigenvalue $\lambda$.
\end{corollary}

\medskip

In Figure~\ref{9eigenvectors} we present all non-zero vectors of the family $\{ E_{11,\lambda} \}_{-3 \leq \lambda \leq n-4}$.

\begin{landscape}
	
	\begin{figure}
		\centering
		{\tiny
			\begin{tabular}{ccc}
				\begin{tabular}{|c|c|c|c|c|c|c|c|c|c|c|}
					\hline \cellcolor{sandstorm}4 & & & & & & & & & & \cellcolor{royalblue}-4 \\
					\hline & \cellcolor{pastelyellow}-1 & & & & & & & & \cellcolor{skyblue}1 & \\
					\hline & & \cellcolor{pastelyellow}-1 & & & & & & \cellcolor{skyblue}1 & & \\
					\hline & & & \cellcolor{pastelyellow}-1 & & & & \cellcolor{skyblue}1 & & & \\
					\hline & & & & \cellcolor{pastelyellow}-1 & & \cellcolor{skyblue}1 & & & & \\
					\hline & & & & & \phantom{-1} & & & & & \\
					\hline & & & & \cellcolor{skyblue}1 & & \cellcolor{pastelyellow}-1 & & & & \\
					\hline & & & \cellcolor{skyblue}1 & & & & \cellcolor{pastelyellow}-1 & & & \\
					\hline & & \cellcolor{skyblue}1 & & & & & & \cellcolor{pastelyellow}-1 & & \\
					\hline & \cellcolor{skyblue}1 & & & & & & & & \cellcolor{pastelyellow}-1 & \\
					\hline \cellcolor{royalblue}-4 & & & & & & & & & & \cellcolor{sandstorm}4 \\
					\hline
				\end{tabular}
				&
				\begin{tabular}{|c|c|c|c|c|c|c|c|c|c|c|}
					\hline & \cellcolor{sandstorm}3 & & & & & & & & \cellcolor{royalblue}-3 & \\
					\hline \cellcolor{sandstorm}3 & & \cellcolor{pastelyellow}-1 & & & & & & \cellcolor{skyblue}1 & & \cellcolor{royalblue}-3 \\
					\hline & \cellcolor{pastelyellow}-1 & & \cellcolor{pastelyellow}-1 & & & & \cellcolor{skyblue}1 & & \cellcolor{skyblue}1 & \\
					\hline & & \cellcolor{pastelyellow}-1 & & \cellcolor{pastelyellow}-1 & & \cellcolor{skyblue}1 & & \cellcolor{skyblue}1 & & \\
					\hline & & & \cellcolor{pastelyellow}-1 & & & & \cellcolor{skyblue}1 & & & \\
					\hline & & & & & \phantom{-1} & & & & & \\
					\hline & & & \cellcolor{skyblue}1 & & & & \cellcolor{pastelyellow}-1 & & & \\
					\hline & & \cellcolor{skyblue}1 & & \cellcolor{skyblue}1 & & \cellcolor{pastelyellow}-1 & & \cellcolor{pastelyellow}-1 & & \\
					\hline & \cellcolor{skyblue}1 & & \cellcolor{skyblue}1 & & & & \cellcolor{pastelyellow}-1 & & \cellcolor{pastelyellow}-1 & \\
					\hline \cellcolor{royalblue}-3 & & \cellcolor{skyblue}1 & & & & & & \cellcolor{pastelyellow}-1 & & \cellcolor{sandstorm}3 \\
					\hline & \cellcolor{royalblue}-3 & & & & & & & & \cellcolor{sandstorm}3 & \\
					\hline
				\end{tabular}
				&
				\begin{tabular}{|c|c|c|c|c|c|c|c|c|c|c|}
					\hline & & \cellcolor{sandstorm}2 & & & & & & \cellcolor{royalblue}-2 & & \\
					\hline & \cellcolor{sandstorm}2 & & \cellcolor{pastelyellow}-1 & & & & \cellcolor{skyblue}1 & & \cellcolor{royalblue}-2 & \\
					\hline \cellcolor{sandstorm}2 & & \cellcolor{pastelyellow}-1 & & \cellcolor{pastelyellow}-1 & & \cellcolor{skyblue}1 & & \cellcolor{skyblue}1 & & \cellcolor{royalblue}-2 \\
					\hline & \cellcolor{pastelyellow}-1 & & \cellcolor{pastelyellow}-1 & & & & \cellcolor{skyblue}1 & & \cellcolor{skyblue}1 & \\
					\hline & & \cellcolor{pastelyellow}-1 & & & & & & \cellcolor{skyblue}1 & & \\
					\hline & & & & & \phantom{-1} & & & & & \\
					\hline & & \cellcolor{skyblue}1 & & & & & & \cellcolor{pastelyellow}-1 & & \\
					\hline & \cellcolor{skyblue}1 & & \cellcolor{skyblue}1 & & & & \cellcolor{pastelyellow}-1 & & \cellcolor{pastelyellow}-1 & \\
					\hline \cellcolor{royalblue}-2 & & \cellcolor{skyblue}1 & & \cellcolor{skyblue}1 & & \cellcolor{pastelyellow}-1 & & \cellcolor{pastelyellow}-1 & & \cellcolor{sandstorm}2 \\
					\hline & \cellcolor{royalblue}-2 & & \cellcolor{skyblue}1 & & & & \cellcolor{pastelyellow}-1 & & \cellcolor{sandstorm}2 & \\
					\hline & & \cellcolor{royalblue}-2 & & & & & & \cellcolor{sandstorm}2 & & \\
					\hline
				\end{tabular}
				\\ \\
				\begin{tabular}{|c|c|c|c|c|c|c|c|c|c|c|}
					\hline & & & \cellcolor{sandstorm}1 & & & & \cellcolor{royalblue}-1 & & & \\
					\hline & & \cellcolor{sandstorm}1 & & \cellcolor{pastelyellow}-1 & & \cellcolor{skyblue}1 & & \cellcolor{royalblue}-1 & & \\
					\hline & \cellcolor{sandstorm}1 & & \cellcolor{pastelyellow}-1 & & & & \cellcolor{skyblue}1 & & \cellcolor{royalblue}-1 & \\
					\hline \cellcolor{sandstorm}1 & & \cellcolor{pastelyellow}-1 & & & & & & \cellcolor{skyblue}1 & & \cellcolor{royalblue}-1 \\
					\hline & \cellcolor{pastelyellow}-1 & & & & & & & & \cellcolor{skyblue}1 & \\
					\hline & & & & & \phantom{-1} & & & & & \\
					\hline & \cellcolor{skyblue}1 & & & & & & & & \cellcolor{pastelyellow}-1 & \\
					\hline \cellcolor{royalblue}-1 & & \cellcolor{skyblue}1 & & & & & & \cellcolor{pastelyellow}-1 & & \cellcolor{sandstorm}1 \\
					\hline & \cellcolor{royalblue}-1 & & \cellcolor{skyblue}1 & & & & \cellcolor{pastelyellow}-1 & & \cellcolor{sandstorm}1 & \\
					\hline & & \cellcolor{royalblue}-1 & & \cellcolor{skyblue}1 & & \cellcolor{pastelyellow}-1 & & \cellcolor{sandstorm}1 & & \\
					\hline & & & \cellcolor{royalblue}-1 & & & & \cellcolor{sandstorm}1 & & & \\
					\hline
				\end{tabular}
				&
				\begin{tabular}{|c|c|c|c|c|c|c|c|c|c|c|}
					\hline & & & & \cellcolor{royalblue}2 & & \cellcolor{sandstorm}-2 & & & & \\
					\hline & & & \cellcolor{skyblue}1 & & & & \cellcolor{pastelyellow}-1 & & & \\
					\hline & & \cellcolor{skyblue}1 & & \cellcolor{celadon}-1 & & \cellcolor{darkcyan}1 & & \cellcolor{pastelyellow}-1 & & \\
					\hline & \cellcolor{skyblue}1 & & \cellcolor{celadon}-1 & & & & \cellcolor{darkcyan}1 & & \cellcolor{pastelyellow}-1 & \\
					\hline \cellcolor{royalblue}2 & & \cellcolor{celadon}-1 & & & & & & \cellcolor{darkcyan}1 & & \cellcolor{sandstorm}-2 \\
					\hline & & & & & \phantom{-1} & & & & & \\
					\hline \cellcolor{sandstorm}-2 & & \cellcolor{darkcyan}1 & & & & & & \cellcolor{celadon}-1 & & \cellcolor{royalblue}2 \\
					\hline & \cellcolor{pastelyellow}-1 & & \cellcolor{darkcyan}1 & & & & \cellcolor{celadon}-1 & & \cellcolor{skyblue}1 & \\
					\hline & & \cellcolor{pastelyellow}-1 & & \cellcolor{darkcyan}1 & & \cellcolor{celadon}-1 & & \cellcolor{skyblue}1 & & \\
					\hline & & & \cellcolor{pastelyellow}-1 & & & & \cellcolor{skyblue}1 & & & \\
					\hline & & & & \cellcolor{sandstorm}-2 & & \cellcolor{royalblue}2 & & & & \\
					\hline
				\end{tabular}
				&
				\begin{tabular}{|c|c|c|c|c|c|c|c|c|c|c|}
					\hline & & & \cellcolor{royalblue}3 & & & & \cellcolor{sandstorm}-3 & & & \\
					\hline & & \cellcolor{skyblue}1 & & \cellcolor{royalblue}3 & & \cellcolor{sandstorm}-3 & & \cellcolor{pastelyellow}-1 & & \\
					\hline & \cellcolor{skyblue}1 & & \cellcolor{skyblue}1 & & & & \cellcolor{pastelyellow}-1 & & \cellcolor{pastelyellow}-1 & \\
					\hline \cellcolor{royalblue}3 & & \cellcolor{skyblue}1 & & \cellcolor{celadon}-2 & & \cellcolor{darkcyan}2 & & \cellcolor{pastelyellow}-1 & & \cellcolor{sandstorm}-3 \\
					\hline & \cellcolor{royalblue}3 & & \cellcolor{celadon}-2 & & & & \cellcolor{darkcyan}2 & & \cellcolor{sandstorm}-3 & \\
					\hline & & & & & \phantom{-1} & & & & & \\
					\hline & \cellcolor{sandstorm}-3 & & \cellcolor{darkcyan}2 & & & & \cellcolor{celadon}-2 & & \cellcolor{royalblue}3 & \\
					\hline \cellcolor{sandstorm}-3 & & \cellcolor{pastelyellow}-1 & & \cellcolor{darkcyan}2 & & \cellcolor{celadon}-2 & & \cellcolor{skyblue}1 & & \cellcolor{royalblue}3 \\
					\hline & \cellcolor{pastelyellow}-1 & & \cellcolor{pastelyellow}-1 & & & & \cellcolor{skyblue}1 & & \cellcolor{skyblue}1 & \\
					\hline & & \cellcolor{pastelyellow}-1 & & \cellcolor{sandstorm}-3 & & \cellcolor{royalblue}3 & & \cellcolor{skyblue}1 & & \\
					\hline & & & \cellcolor{sandstorm}-3 & & & & \cellcolor{royalblue}3 & & & \\
					\hline
				\end{tabular} \\ \\
				\begin{tabular}{|c|c|c|c|c|c|c|c|c|c|c|}
					\hline & & \cellcolor{royalblue}4 & & & & & & \cellcolor{sandstorm}-4 & & \\
					\hline & \cellcolor{skyblue}1 & & \cellcolor{royalblue}4 & & & & \cellcolor{sandstorm}-4 & & \cellcolor{pastelyellow}-1 & \\
					\hline \cellcolor{royalblue}4 & & \cellcolor{skyblue}1 & & \cellcolor{royalblue}4 & & \cellcolor{sandstorm}-4 & & \cellcolor{pastelyellow}-1 & & \cellcolor{sandstorm}-4 \\
					\hline & \cellcolor{royalblue}4 & & \cellcolor{skyblue}1 & & & & \cellcolor{pastelyellow}-1 & & \cellcolor{sandstorm}-4 & \\
					\hline & & \cellcolor{royalblue}4 & & \cellcolor{celadon}-3 & & \cellcolor{darkcyan}3 & & \cellcolor{sandstorm}-4 & & \\
					\hline & & & & & \phantom{-1} & & & & & \\
					\hline & & \cellcolor{sandstorm}-4 & & \cellcolor{darkcyan}3 & & \cellcolor{celadon}-3 & & \cellcolor{royalblue}4 & & \\
					\hline & \cellcolor{sandstorm}-4 & & \cellcolor{pastelyellow}-1 & & & & \cellcolor{skyblue}1 & & \cellcolor{royalblue}4 & \\
					\hline \cellcolor{sandstorm}-4 & & \cellcolor{pastelyellow}-1 & & \cellcolor{sandstorm}-4 & & \cellcolor{royalblue}4 & & \cellcolor{skyblue}1 & & \cellcolor{royalblue}4 \\
					\hline & \cellcolor{pastelyellow}-1 & & \cellcolor{sandstorm}-4 & & & & \cellcolor{royalblue}4 & & \cellcolor{skyblue}1 & \\
					\hline & & \cellcolor{sandstorm}-4 & & & & & & \cellcolor{royalblue}4 & & \\
					\hline
				\end{tabular}
				&
				\begin{tabular}{|c|c|c|c|c|c|c|c|c|c|c|}
					\hline & \cellcolor{royalblue}5 & & & & & & & & \cellcolor{sandstorm}-5 & \\
					\hline \cellcolor{royalblue}5 & & \cellcolor{royalblue}5 & & & & & & \cellcolor{sandstorm}-5 & & \cellcolor{sandstorm}-5 \\
					\hline & \cellcolor{royalblue}5 & & \cellcolor{royalblue}5 & & & & \cellcolor{sandstorm}-5 & & \cellcolor{sandstorm}-5 & \\
					\hline & & \cellcolor{royalblue}5 & & \cellcolor{royalblue}5 & & \cellcolor{sandstorm}-5 & & \cellcolor{sandstorm}-5 & & \\
					\hline & & & \cellcolor{royalblue}5 & & & & \cellcolor{sandstorm}-5 & & & \\
					\hline & & & & & \phantom{-1} & & & & & \\
					\hline & & & \cellcolor{sandstorm}-5 & & & & \cellcolor{royalblue}5 & & & \\
					\hline & & \cellcolor{sandstorm}-5 & & \cellcolor{sandstorm}-5 & & \cellcolor{royalblue}5 & & \cellcolor{royalblue}5 & & \\
					\hline & \cellcolor{sandstorm}-5 & & \cellcolor{sandstorm}-5 & & & & \cellcolor{royalblue}5 & & \cellcolor{royalblue}5 & \\
					\hline \cellcolor{sandstorm}-5 & & \cellcolor{sandstorm}-5 & & & & & & \cellcolor{royalblue}5 & & \cellcolor{royalblue}5 \\
					\hline & \cellcolor{sandstorm}-5 & & & & & & & & \cellcolor{royalblue}5 & \\
					\hline
				\end{tabular}
				&
				\begin{tabular}{|c|c|c|c|c|c|c|c|c|c|c|}
					\hline \cellcolor{royalblue}6 & & & & & & & & & & \cellcolor{sandstorm}-6 \\
					\hline & \cellcolor{royalblue}6 & & & & & & & & \cellcolor{sandstorm}-6 & \\
					\hline & & \cellcolor{royalblue}6 & & & & & & \cellcolor{sandstorm}-6 & & \\
					\hline & & & \cellcolor{royalblue}6 & & & & \cellcolor{sandstorm}-6 & & & \\
					\hline & & & & \cellcolor{royalblue}6 & & \cellcolor{sandstorm}-6 & & & & \\
					\hline & & & & & \phantom{-1} & & & & & \\
					\hline & & & & \cellcolor{sandstorm}-6 & & \cellcolor{royalblue}6 & & & & \\
					\hline & & & \cellcolor{sandstorm}-6 & & & & \cellcolor{royalblue}6 & & & \\
					\hline & & \cellcolor{sandstorm}-6 & & & & & & \cellcolor{royalblue}6 & & \\
					\hline & \cellcolor{sandstorm}-6 & & & & & & & & \cellcolor{royalblue}6 & \\
					\hline \cellcolor{sandstorm}-6 & & & & & & & & & & \cellcolor{royalblue}6 \\
					\hline
				\end{tabular}
			\end{tabular}
		}
		\caption{First row, from left to right: $E_{11,-3}$, $E_{11,-2}$, and $E_{11,-1}$. Second row, from left to right: $E_{11,0}$, $E_{11,3}$, and $E_{11,4}$. Third row, from left to right: $E_{11,5}$, $E_{11,6}$, and $E_{11,7}$.}
		\label{9eigenvectors}
	\end{figure}
	
\end{landscape}

\section{The eigenvalue $n-4$}\label{Sec-n-4}

The eigenvalue $n-4$ of $\mathcal{Q}(n)$ also appears in Section~\ref{Sec-SequenceEigenvalues} for every odd $n \ge 4$. Now we present a family of eigenvectors of $\mathcal{Q}(n)$ associated with $n-4$ and establish a lower bound for its multiplicity for $n \ge 3$ not necessarily odd.

Given $n \geq 3$ and $1 \leq \ell \leq n$, let $C_{n,\ell}$, $R_{n,\ell}$ and $F_{n, \ell}$ be the vectors defined as follows.

\begin{align*}
	(C_{n,\ell})_{(i,j)} & =
	\begin{cases}
		1, & \text{if $j = \ell$ or $j = n+1 - \ell$}; \\
		0, & \text{otherwise}
	\end{cases}\\
	(R_{n,\ell})_{(i,j)} & =
	\begin{cases}
		-1, & \text{if $i = \ell$ or $i = n+1 - \ell$}; \\
		0, & \text{otherwise}
	\end{cases}
	\intertext{and}
	F_{n, \ell} & = C_{n,\ell} + R_{n,\ell}
\end{align*}

In Figure \ref{fig:CReF} we display $C_{5,2}$, $R_{5,2}$ and $F_{5,2}$.

\begin{figure}[hhh]
	\centering
	{\scriptsize
			\begin{tabular}{|c|c|c|c|c|}
				\hline \hphantom{-1} & \cellcolor{skyblue}\hphantom{-}1 & \hphantom{-1} & \cellcolor{skyblue}\hphantom{-}1 & \hphantom{-1} \\
				\hline & \cellcolor{skyblue}1 & & \cellcolor{skyblue}1 & \\
				\hline & \cellcolor{skyblue}1 & & \cellcolor{skyblue}1 & \\
				\hline & \cellcolor{skyblue}1 & & \cellcolor{skyblue}1 & \\
				\hline & \cellcolor{skyblue}1 & & \cellcolor{skyblue}1 & \\
                \hline
			\end{tabular} \quad + \quad
			\begin{tabular}{|c|c|c|c|c|}
				\hline & & & & \\
				\hline \cellcolor{pastelyellow}-1 & \cellcolor{pastelyellow}-1 & \cellcolor{pastelyellow}-1 & \cellcolor{pastelyellow}-1 & \cellcolor{pastelyellow}-1 \\
				\hline & & & & \\
				\hline \cellcolor{pastelyellow}-1 & \cellcolor{pastelyellow}-1 & \cellcolor{pastelyellow}-1 & \cellcolor{pastelyellow}-1 & \cellcolor{pastelyellow}-1 \\
				\hline & & & & \\
				\hline
			\end{tabular} \quad = \quad
			\begin{tabular}{|c|c|c|c|c|}
				\hline & \cellcolor{skyblue}1 & & \cellcolor{skyblue}1 & \\
				\hline \cellcolor{pastelyellow}-1 & \hphantom{-1} & \cellcolor{pastelyellow}-1 & \hphantom{-1} & \cellcolor{pastelyellow}-1 \\
				\hline & \cellcolor{skyblue}1 & & \cellcolor{skyblue}1 & \\
				\hline \cellcolor{pastelyellow}-1 & \hphantom{-1} & \cellcolor{pastelyellow}-1 & \hphantom{-1} & \cellcolor{pastelyellow}-1 \\
				\hline & \cellcolor{skyblue}1 & & \cellcolor{skyblue}1 & \\
				\hline
			\end{tabular}
	}
	\caption{$C_{5,2}$, $R_{5,2}$, and $F_{5,2}$}
	\label{fig:CReF}
\end{figure}

Clearly we have $C_{n,\ell} = C_{n,n+1-\ell}$, $R_{n,\ell} = R_{n,n+1-\ell}$, and $F_{n,\ell} = F_{n,n+1-\ell}$, for every $1 \leq \ell \leq n$. Note also that $F_{n,1} + F_{n,2} + \cdots + F_{n, \left\lceil \frac{n}{2} \right\rceil} = 0$ and so $\left\{ F_{n,1}, F_{n,2}, \ldots, F_{n, \left\lceil \frac{n}{2} \right\rceil} \right\}$ is not linearly independent. However, considering the vectors $E_{n,n-4}$ obtained from the family defined in \eqref{family_vectors} for odd $n \ge 4$, it is immediate that $E_{n,n-4}$ can be defined also for $n=3$ and it is an eigenvector. Furthermore, we have the following result.

\begin{theorem} Let $n \geq 3$.
\begin{enumerate}
\item If $n$ is even, $\left\{ F_{n,1}, \ldots, F_{n,\frac{n-2}{2}} \right\}$ is a linearly independent set of $\frac{n-2}{2}$ eigenvectors associated with $n-4$.
\item If $n$ is odd, $\left\{ E_{n,n-4}, F_{n,1}, \ldots, F_{n,\frac{n-1}{2}} \right\}$ is a linearly independent set of $\frac{n+1}{2}$ eigenvectors associated with $n-4$.
\end{enumerate}
\end{theorem}

\begin{proof}
In what follows we only consider the case where $n \geq 3$ is an odd integer. The case where $n \ge 3$ is even is similar.

First we note that $F_{n,\ell} \in \mathcal{E}_{\mathcal{Q}(n)} (n-4)$, for every $1 \leq \ell \leq n$. Let $(p,q) \in [n]^2$. Then

\begin{align*}
\sum_{j=1}^n (C_{n,\ell})_{(p,j)} & = - \sum_{i=1}^n (R_{n,\ell})_{(i,q)} = 2 \\
\sum_{i=1}^n (C_{n,\ell})_{(i,q)} & = \begin{cases}
                                                    n & \text{if $q = \ell$ or $q = n+1 - \ell$;} \\
                                                    0 & \text{otherwise;}
                                                   \end{cases} \\
\sum_{j=1}^n (R_{n,\ell})_{(p,j)} & = \begin{cases}
                                       -n & \text{if $p = \ell$ or $p = n+1 - \ell$;} \\
                                        0 & \text{otherwise;}
                                      \end{cases}.
\end{align*}
In addition, since $(R_{n,\ell})_{(i,j)} = - (C_{n,\ell})_{(j,i)} = - (C_{n,\ell})_{(n+1-j,n+1-i)}$,
\begin{align*}
\sum_{i+j=p+q} (F_{n,\ell})_{(i,j)} & = \sum_{i+j=p+q}^n (R_{n,\ell})_{(i,j)} + \sum_{i+j=p+q}^n (C_{n,\ell})_{(i,j)} = 0, \\
\sum_{i-j=p-q} (F_{n,\ell})_{(i,j)} & = \sum_{i-j=p-q}^n (R_{n,\ell})_{(i,j)} + \sum_{i-j=p-q}^n (C_{n,\ell})_{(i,j)} = 0.
\end{align*}
Now,
\begin{align*}
((A + 4 I_n) F_{n,\ell})_{(p,q)} & = \sum_{i=1}^n (F_{n,\ell})_{(i,q)} + \sum_{j=1}^n (F_{n,\ell})_{(p,j)} +
\sum_{i+j=p+q} (F_{n,\ell})_{(i,j)} + \sum_{i-j=p-q} (F_{n,\ell})_{(i,j)} \\
& = \sum_{i=1}^n (F_{n,\ell})_{(i,q)} + \sum_{j=1}^n (F_{n,\ell})_{(p,j)} \\
& = \begin{cases}
(n-2)+(2-n)= 0 & \text{if $p, q \in \{ \ell, n+1 - \ell \}$;} \\
(0-2)+(2-n) = -n & \text{if $p \in \{ \ell, n+1 - \ell \}$ and $q \notin \{ \ell, n+1 - \ell \}$;} \\
(n-2)+(2+0) = n & \text{if $p \notin \{ \ell, n+1 - \ell \}$ and $q \in \{ \ell, n+1 - \ell \}$;} \\
(0-2)+(2+0) = 0 & \text{if $p, q \notin \{ \ell, n+1 - \ell \}$;}
\end{cases} \\
& = n(F_{n,\ell})_{(p,q)},
\end{align*}
so $F_{n,\ell}$ is an eigenvector of $\mathcal{Q}(n)$ associated with $n-4$.
	
Finally we prove that $\left\{ E_{n,n-4}, F_{n,1}, \ldots, F_{n,\frac{n-1}{2}} \right\}$ is linearly independent.

Let $\alpha_0, \alpha_1, \ldots, \alpha_{\frac{n-1}{2}} \in \mathbb{R}$ be such that
\[
L := \alpha_0 E_{n,n-4} + \alpha_1 F_{n,1} + \cdots + \alpha_{\frac{n-1}{2}} F_{n,\frac{n-1}{2}} = 0.
\]
Then $\alpha_0 = L_{(1,1)} = 0$ and
$\alpha_j = L_{\frac{n-1}{2},j} = 0$, for all $1 \leq j \leq \frac{n-1}{2}$.
\end{proof}

\begin{corollary} Let $n \geq 3$.
\begin{enumerate}
\item If $n$ is even, $n-4$ is an eigenvalue of $\mathcal{Q}(n)$ with multiplicity at least $\frac{n-2}{2}$.
\item If $n$ is odd, $n-4$ is an eigenvalue of $\mathcal{Q}(n)$ with multiplicity at least $\frac{n+1}{2}$.
\end{enumerate}
\end{corollary}

\section{Open problems}\label{Sec-OpenProblems}

Denoting the set of distinct eigenvalues of $\mathcal{Q}(n)$ by $\sigma(\mathcal{Q}(n))$, the following conclusion was deduced.

The integer $n-4$ is an eigenvalue of $\mathcal{Q}(3)$ and, assuming that $n \ge 4$,

\begin{equation}\label{eq-openproblem}
	\sigma(\mathcal{Q}(n)) \cap \mathbb{Z} \supseteq
	\begin{cases}
		\left\{ -4, n-4 \right\}, & \text{if $n$ is even;} \\
		\left\{ -4, -3, \ldots, \frac{n-11}{2} \right\} \cup \left\{ \frac{n-5}{2}, \ldots, n-5, n-4 \right\}, & \text{if $n$ is odd.}
	\end{cases}
\end{equation}

However, the exact characterization of the integer eigenvalues of $\mathcal{Q}(n)$ remains as an open problem. In particular, the eigenvalues of $\mathcal{Q}(n)$, computed for several values of $n$, suggest the following conjecture.

\begin{conjecture}
For $n \ge 4$, the integers which appear in \eqref{eq-openproblem} are the unique integer eigenvalues of $\mathcal{Q}(n)$. Furthermore, when $n$ is even the eigenvalue $n-4$ has multiplicity $\frac{n-2}{2}$, and when $n$ is odd the eigenvalue $n-4$ has multiplicity $\frac{n+1}{2}$. The eigenvalues $-3, -2, \dots, \frac{n-11}{2}, \frac{n-5}{2}, \dots, n-6, n-5$ are simple.
\end{conjecture}

From the computations we may say that this conjecture is true for $n \le 100$.  If this conjecture is true, then $\mathcal{Q}(n)$ has merely two or $n-1$ distinct integer eigenvalues when $n$ is even or odd, respectively.

\bigskip

\noindent \textbf{Acknowledgments.}
This work is supported by the Center for Research and Development in Mathematics and Applications (CIDMA) through the Portuguese Foundation for Science and Technology (FCT - Fundação para a Ciência e a Tecnologia), reference UIDB/04106/2020. I.S.C. also thanks the support of FCT - Fundação para a Ciência e a Tecnologia via the Ph.D. Scholarship PD/BD/150538/2019.

\end{document}